\documentclass[11pt,a4paper]{amsart}

\usepackage{amsmath}
\usepackage{amssymb}
\usepackage{comment}
\usepackage{fullpage}
\usepackage[all]{xy}

\usepackage[square,numbers]{natbib}

\usepackage{color}

\usepackage{amsthm}   
\theoremstyle{plain}
\newtheorem{theorem}{Theorem}[section]
\newtheorem{lemma}[theorem]{Lemma}
\newtheorem{corollary}[theorem]{Corollary}
\newtheorem{proposition}[theorem]{Proposition}
\theoremstyle{definition}

\newtheorem{example}[theorem]{Example}
\theoremstyle{remark}
\newtheorem{remark}[theorem]{Remark}

\newcommand{\parantezicindeenumi}[1][\roman]{\renewcommand{\theenumi}{#1{enumi}}
\renewcommand{\labelenumi}{(\theenumi)} }

\newcommand{\Z}{{\mathbb{Z}}}

\newcommand{\Q}{{\mathbb{Q}}}

\newcommand{\Ker}{\operatorname{Ker}} 
\newcommand{\Image}{\operatorname{Im}}
\newcommand{\Coker}{\operatorname{Coker}}   
\newcommand{\Soc}{\operatorname{Soc}} 
\newcommand{\Rad}{\operatorname{Rad}}
\newcommand{\Jac}{\operatorname{Jac}} 
\newcommand{\Ext}{\operatorname{Ext}} 
\newcommand{\Tor}{\operatorname{Tor}}

\newcommand{\Hom}{\operatorname{Hom}}

\newcommand{\End}{\operatorname{End}}

\newcommand{\udim}{\operatorname{u.dim}} 
\newcommand{\hdim}{\operatorname{h.dim}}

\newcommand{\Tr}{\operatorname{Tr}}
\newcommand{\pd}{\operatorname{pd}}

\newcommand{\Transpose}{\operatorname{Tr}}

\newcommand{\ShortExactSequence}[5]{{\xymatrix@1{ 0 \ar[r] &  #1 \ar[r]^-{#2} & #3 \ar[r]^-{#4} &  #5  \ar[r] & 0 }}}

\newcommand{\AShortExactSequence}[6]{{\xymatrix@1{ #1: & 0 \ar[r] &  #2 \ar[r]^-{#3} & #4 \ar[r]^-{#5} &  #6  \ar[r] & 0 }}}

\newcommand{\Endomorphism}{\operatorname{End}}

\newcommand{\rann}{\operatorname{r.ann}}

\usepackage{hyperref}

\begin{document}

\title{DECOMPOSITIONS INTO A DIRECT SUM OF \\ PROJECTIVE AND STABLE SUBMODULES}

\date{01.04.2026}

\author {G\"{U}L\.{I}ZAR G\"{U}NAY$^*$ AND 
ENG{\.{I}}N MERMUT}
\address{Dokuz Eyl\"ul \"Universitesi, T{\i}naztepe Yerle\c skesi, Fen Fak\"ultesi, Matematik B\"ol\"um\"u,  \mbox{ } 35390\\  
Buca/Izmir, Turkey}
\address{Graz University of Technology, Rechbauerstraße 12, 8010 Graz, Austria}

\email{gunaymert@math.tugraz.at}
\email{engin.mermut@deu.edu.tr}


\begin{abstract}

A module $M$ is {called} stable if it has no nonzero projective direct summand.
For a ring $ R $, we study conditions under which $R$-modules from certain classes decompose as a direct sum of a projective submodule and a stable submodule.   
Over {an arbitrary} ring, modules of finite uniform dimension or finite hollow dimension can be decomposed as a direct sum of a projective submodule and a stable submodule. By using the Auslander-Bridger transpose of finitely presented modules, we prove that every finitely presented right $R$-module over a left semihereditary ring $R$ has such a decomposition. 
Our main focus in this article is to give examples where such a decomposition fails.
We give some ring examples over which there exists an infinitely generated or finitely generated or finitely presented module where such a decomposition fails. Our main example is a cyclically presented module $M$ over a commutative ring such that~$M$ has no such decomposition and $M$ is not projectively equivalent to a stable module.

\end{abstract}

\maketitle

\begingroup
\renewcommand{\thefootnote}{*}
\footnotetext{This research was funded in whole or in part by the Austrian Science Fund (FWF) 10.55776/ESP4162324. For open access purposes, the author has applied a CC BY public copyright license to any author-accepted manuscript version arising from this submission.}
\endgroup

\section{Introduction}

Let $R$ be an arbitrary \emph{associative} ring with unity. An $R$-module or module means a unital \emph{right} $R$-module unless otherwise stated.

Following the terminology in~\cite{Martsinkovsky:1TorsionOfFiniteModulesOverSemiperfectRings, Martsinkovsky:TheStableCategoryOfALeftHereditaryRing} and  \cite{Zangurashvili:AStructureTheoremForLeftModulesOverLeftHereditaryLeftPerfectRightCoherentRings}, {a module $M$ is called \emph{stable} if it has no nonzero projective direct summand. Dually, a module is called \emph{costable} if it has no nonzero injective direct summand (equivalently, if it has no nonzero injective submodule).}
In~\cite{HeZhengXu:CharacterizationsOfNoetherianAndHereditaryRings}, He characterized left Noetherian rings {as rings over which every left module decomposes as a direct sum of an injective submodule and a costable submodule}. Moreover, He showed that a ring $ R $ is left Noetherian and left hereditary if and only if every left $ R $-module $ M $ decomposes as a direct sum of an injective submodule and a costable submodule and for all decompositions  $ M=D\oplus B = D'\oplus B'$, where $ D $ and $ D' $ are injective submodules, $ B $ and $ B' $ are costable submodules of $ M $, we have $ D=D' $~\cite[Theorem 2]{HeZhengXu:CharacterizationsOfNoetherianAndHereditaryRings}
 (which obviously implies that $ B $ is isomorphic to $ B' $; see also~\cite{Zangurashvili:AStructureTheoremForLeftModulesOverLeftHereditaryLeftPerfectRightCoherentRings}). Our interest is in the dual problem: examining examples where modules from a specific class decompose as a direct sum of a projective module and a stable module, or where such decompositions fail.
The first result along this line was obtained by Warfield who proved that any finitely generated module $ M $ over a semiperfect ring has a decomposition $M=P \oplus N$ for some projective submodule $P$ and a stable submodule $N$ of $M$,  and   if $M=P^{\prime} \oplus N^{\prime}$ is another such decomposition, then $P \cong P^{\prime}$ and $N \cong N^{\prime}$~\cite[Theorem 1.4]{Warfield:SerialRingsAndFinitelyPresentedModules}; see also~\cite[Theorem 3.15]{Facchini:ModuleTheory}. In~\cite{Zangurashvili:AStructureTheoremForLeftModulesOverLeftHereditaryLeftPerfectRightCoherentRings}, using a categorical approach, Zangurashvili proves that for a left hereditary ring, every left module has a decomposition into the direct sum of a projective module and a stable module if and only if the ring is left perfect and right coherent. She also shows that, in this case,
 if $M=P \oplus N=P^{\prime} \oplus N^{\prime}$ with projective submodules $P$ and $P^{\prime}$ and stable submodules $N$ and $N^{\prime}$ of a module $ M $, then $N$ is equal to $N^{\prime}$ and $P$ is isomorphic to $P^{\prime}$.

In Section~\ref{Sec:Decompositions_of_Modules_with_Finite_Uniform_Hollow_Dimensions}, we shall see that over any ring, modules of finite uniform dimension or finite hollow dimension decompose as a direct sum of a projective submodule and a stable submodule.
This implies that such a decomposition holds for all finitely generated modules over a semilocal ring. Clearly, such a decomposition holds for Noetherian or Artinian modules (and so for finitely generated modules over a right Noetherian ring).

In Section~\ref{Sec:Decompositions_over_Semihereditary_Rings}, we shall employ the Auslander-Bridger transpose of finitely presented modules to prove that if
 $R$ is a \emph{left} semihereditary ring and $M$ is a finitely presented (\emph{right}) $R$-module, then $M$ has a decomposition $M=P\oplus N$ with a projective submodule $P$ and a stable submodule $N$ of $M$ (Theorem~\ref{Ex:left semihereditary right module}). 

Our main focus in this article is on the examples of modules where such a decomposition fails. In Section \ref{Sec:Examples_where_the_decomposition_fails},  we give examples of rings {and infinitely generated or finitely generated or finitely presented modules over these rings for which the decomposition fails}. Our {final example} is a finitely presented module $M$ (in fact a cyclically presented module) over a commutative ring such that $M$ has no such decomposition and $M$ is not projectively equivalent to a stable module (Example~\ref{Ex:Cyclically_presented_module}). {Recall} that modules $A$ and $B$ are said to be \emph{projectively equivalent} if there exist projective modules $P$ and $Q$ such that $A\oplus P \cong B\oplus Q$. Note that this means that the modules $A$ and $B$ are isomorphic objects in the stable category of $ R $-modules. The existence of a module $ M $'s decomposition $ M =P\oplus N$ with projective  $ P $  and stable $ N $, enables one to deal with the stable module $ N $ instead of $ M $ in the stable category. For finitely generated modules over Artin algebras (or more generally over semiperfect rings),  such an approach is used in the representation theory of algebras; see~\cite[p. 104-105]{Auslander-Reiten-Smalo:RepresentationTheoryOfArtinAlgebras}. Over any ring~$R$, Facchini and Girardi consider in~\cite{Facchini-Girardi:AuslanderBridgerModules} some classes of finitely generated $R$-modules or finitely presented $R$-modules such that every module within each class decomposes, uniquely up to isomorphism, as a direct sum of a projective submodule and a stable submodule from that class. The examples we give demonstrate the cases where such a decomposition may fail. There is no stable module in the stable isomorphism class of the module in our last Example~\ref{Ex:Cyclically_presented_module}. The authors are grateful to Noyan Er for discussions about the problems considered in this paper; in particular, the examples in Theorem \ref{Ex:Cyclicmodule} and Example~\ref{Ex:Cyclically_presented_module} have been found by him.

The terminology and notation that will be used throughout the paper are as follows. For rings $R$ and $S$, ${ }_S M_R$ denotes an $S$-$R$-bimodule,  $M_R$ denotes a (right) $R$-module, ${}_R M$ a left $R$-module (and $R$-$R$-bimodule ${}_R M_R$ is called $R$-bimodule); for the ring $R$, we write $R_R$ (resp., ${}_R R$ and ${}_R R_R$) when considering it as a right $R$-module (resp., left $R$-module and $R$-bimodule). The following definitions for $R$-modules are also similarly given for left $R$-modules. For an $R$-module~$M$, $\Rad(M)$ denotes the \emph{radical} of $M$, that is, the intersection of all maximal submodules of $M$, and $ \Soc(M) $ denotes the \emph{socle} of $ M $, that is, the sum of all simple submodules of $ M $. $\Jac(R)$ denotes the Jacobson radical of the ring $ R $. The \emph{projective dimension} of a module $ M $ is denoted by $ \pd(M) $. An $ R $-module $ M $ is said to be \emph{finitely presented} if it is isomorphic to the cokernel of a module homomorphism $ f:R^n \to R^m $ for some positive integers $ n,m $. A cyclic $R$-module $M$ is called \emph{cyclically presented} if $M_R\cong R/I$ where $ I $ is a principal right ideal of~$R$. A ring $ R $ is said to be \emph{right coherent} if every finitely generated submodule of the right $R$-module $ R_R $ is finitely presented, equivalently, every finitely presented (right) $ R $-module~$ M $ is a \emph{coherent module} which means that every finitely generated submodule of $ M $ is finitely presented; similarly \emph{left coherent} rings are defined, see~\cite[\S 4G]{Lam:LecturesOnModulesAndRings}. A ring~$R$ is said to be \emph{local} (resp., \emph{semilocal}) if $R$ has a unique maximal right ideal (resp., $R/\Jac(R)$ is a semisimple ring). A ring $R$ is said to be \emph{semiprimary} if $R$ is semilocal and $\Jac(R)$ is a nilpotent ideal. A \emph{projective cover} of a module $ M $ is an epimorphism $ P\to M $ whose kernel is a superfluous submodule of $ P $ and whose domain $ P $ is projective; see the next section for the definition of superfluous submodules. A ring~$R$ is said to be \emph{right} (resp., \emph{left}) \emph{perfect}  if every right (resp. left) $ R $-module  has a projective cover. A ring $R$ is said to be \emph{semiperfect} if $R$ is semilocal and idempotents of $R/\Jac(R)$ can be lifted to $R$ (that is,  for every idempotent $a \in R/\Jac(R)$, there exists an idempotent $e\in R$ such that $a=e+ \Jac(R)$). A ring $R$ is called \emph{right hereditary} (resp., \emph{right semihereditary}) if every right ideal (resp., finitely generated right ideal) of $R$ is projective. Similarly, \emph{left hereditary} and \emph{left semihereditary} rings are defined. For {other definitions}, {we refer the reader to}~\cite{Anderson-Fuller:RingsandCategoriesofModules, Facchini:ModuleTheory, Kasch:ModulesAndRings,
Lam:LecturesOnModulesAndRings, Lam:AFirstCourseInNoncommutativeRingsSecondEdition}.

 \section{Decompositions of Modules with Finite Uniform/Hollow Dimensions}\label{Sec:Decompositions_of_Modules_with_Finite_Uniform_Hollow_Dimensions}

Let $(L, \vee, \wedge)$ be a \emph{modular lattice} with 0 and 1, that is, a lattice with a smallest element 0 and a greatest element 1 such that $a \wedge(b \vee c)=(a \wedge b) \vee c$ for every $a, b, c \in L$ with $c \leq a$. An element $a \in L$ is said to be \emph{essential} if $a \wedge x \neq 0$ for every nonzero element $x \in L$. A finite subset $\left\{a_i \mid i \in I\right\}$ of $L \backslash\{0\}$ is said to be \emph{join-independent} if $a_i \wedge\left(\bigvee_{j \neq i} a_j\right)=0$ for every $i \in I$. The empty subset of $ L\setminus  \{0\} $ is
join-independent. An arbitrary subset $A$ of $L \backslash\{0\}$ is said to be \emph{join-independent} if all its finite subsets are join-independent. A lattice $L \neq\{0\}$ is said to be \emph{uniform} if all its nonzero elements are essential. An element $a$ of a modular lattice $L$ is said to be \emph{uniform} if $a \neq 0$ and the lattice
$[0, a] =
\{x \in L \mid  0 \leq x \leq a\}$ is uniform. By~\cite[Theorem 2.36]{Facchini:ModuleTheory}, for a nonzero modular lattice $L$, either there is a finite join-independent subset $\left\{a_1, a_2, \ldots, a_n\right\}$ with $a_i$ uniform for every $i=1,2, \ldots, n$ and
$a_1 \vee a_2 \vee \cdots \vee a_n $ essential, and in this case such a positive integer $n$ is unique and it is said to be the \emph{Goldie dimension} of $L$, or $L$ contains infinite join-independent subsets, in which case $L$ is said to have \emph{infinite Goldie dimension}. The Goldie dimension of a lattice $L$ is zero if and only if $L$ has exactly one element.  If ($L, \wedge, \vee$) is a modular lattice with a smallest element and a greatest element, then its dual lattice ($L, \vee, \wedge$) is also a modular lattice with a smallest element and a greatest element.

A module $M$ is said to have the \emph{uniform} (resp. \emph{hollow}) \emph{dimension} $n$, denoted by $ \udim(M)=n $ (resp. $ \hdim(M)=n $), if its submodule lattice $\mathcal{L}(M)$ (resp. the dual of $\mathcal{L}(M)$) has the Goldie dimension $n$. If there exists no positive integer $ n $ such that $\udim(M)=n$ (resp. $\hdim(M)=n$), then we write $\udim(M)=\infty$ (resp. $\hdim(M)=\infty$); otherwise we write $\udim(M)<\infty$ (resp. $\hdim(M)<\infty$). Note that $\udim(M)=\infty$ holds if and only if $M$ contains an infinite direct sum of nonzero submodules. See \cite[Sections 2.6, 2.7, 2.8]{Facchini:ModuleTheory} and \cite[Section 6A]{Lam:LecturesOnModulesAndRings}.

An \emph{essential} (resp. \emph{superfluous}) submodule $K$ of  $M$ is defined to be an essential element in the lattice $\mathcal{L}(M)$ (resp. the dual of $\mathcal{L}(M)$), that is, for every submodule $L$ of $ M$, $K\cap L=0$ implies $L=0$ (resp. $K+L=M$ implies $L=M$). An $R$-module $M$ is said to be \emph{uniform} (resp. \emph{hollow}) if $M\neq 0$ and the lattice $\mathcal{L}(M)$ (resp. the dual of $\mathcal{L}(M)$) is uniform, that is, every nonzero submodule of $ M $ is an  essential submodule of $ M $ (resp. every proper submodule of $M$ is a superfluous submodule of $M$).

\begin{lemma}\label{decompositiontheorem} If a module $M$ cannot be decomposed as $M=P \oplus N$ where $P$ is a projective submodule and $N$ is a stable submodule, then there exists a sequence $(P_{k})_{k=1}^{\infty}$ of nonzero proper projective submodules of $M$ and a sequence $(N_{k})_{k=1}^{\infty}$ of nonzero proper submodules of $M$ such that for every $k\in \Z^{+}$, \[M=N_{k}\oplus P_{k}\oplus P_{k-1}\oplus \cdots \oplus P_{1}\quad \text{with}\quad N_{k}=N_{k+1}\oplus P_{k+1},\] and so $\udim(M)=\infty=\hdim(M)$ and $ M $ contains the infinite direct sum $ \displaystyle\oplus_{k=1}^\infty P_k $ of nonzero projective submodules.
\end{lemma}
\begin{proof}
The module $M$ is clearly neither projective nor stable. So, $M=P_1 \oplus N_1$, for a nonzero projective module $P_1$ and some $N_1$. It is clear that $N_1$ is neither projective nor stable. Hence, we have the decomposition $N=P_2 \oplus N_2$ with a nonzero projective $P_2$ and some $N_2$. Continuing in this way, we obtain the desired sequences $\left(P_k\right)_{k=1}^{\infty}$ and $\left(N_k\right)_{k=1}^{\infty}$.

For every $n \in \Z^{+}$, $ M $ contains the direct sum $ P_n\oplus P_{n-1}\oplus \cdots \oplus P_1 $ of $ n $ nonzero submodules and so $ \udim(M)=\infty $ by~\cite[6.6]{Lam:LecturesOnModulesAndRings}. Suppose $\hdim(M)<\infty$. Then, for any $ k $, the direct summands $ P_k$ and $ N_k $ of $ M $, being epimorphic images of $ M $, also have finite hollow dimension $ \geq 1 $ by~\cite[Proposition 2.42-(a,c)]{Facchini:ModuleTheory}. So $\hdim(M)= \hdim(N_n\oplus P_n\oplus P_{n-1}\oplus \cdots \oplus P_1)=\hdim (N_n)+\sum_{k=1}^n \hdim(P_k)\geq n+1$ for every $ n\in\Z^+ $ by \cite[Proposition 2.42-(e)]{Facchini:ModuleTheory}, contradicting $ \hdim(M)<\infty $. Thus $\hdim(M)=\infty$.
\end{proof}

\begin{theorem}
\label{Cor:Decomposition-hold-for-finite-udim-hdim}
If $\udim(M)<\infty $ or $\hdim(M)< \infty$ for a module $ M $, then $M$ can be decomposed as $M=P \oplus N$
for some projective submodule $ P $ and stable submodule $ N $ of $ M $.
\end{theorem}

Clearly a Noetherian or an Artinian module cannot contain an infinite direct sum of nonzero projective submodules. Hence Theorem~\ref{Cor:Decomposition-hold-for-finite-udim-hdim} implies

\begin{corollary}\label{Cor:Noetherian-Artinian-Modules-Decompose}
Any module that is either Noetherian or Artinian has a decomposition into the direct sum of a projective submodule and a stable submodule. In particular, a finitely generated module over a right Noetherian ring has such a decomposition.
\end{corollary}

Note that, for the particular case where a ring is right Artinian, the second part of Corollary~\ref{Cor:Noetherian-Artinian-Modules-Decompose} is well-known (see, e.g.,~\cite[p. 104, after Proposition 1.6]{Auslander-Reiten-Smalo:RepresentationTheoryOfArtinAlgebras}).

{\begin{corollary} \label{corollary 3.2.3} If $R$ is a semilocal ring and $ M $ is a finitely generated $R$-module, then   $M=P\oplus N$ for a projective submodule $P$ and a stable submodule $N$.
\end{corollary}}
\begin{proof}
Semilocal rings are exactly the rings with finite hollow dimension as right or left modules over themselves (\cite[Proposition 2.43]{Facchini:ModuleTheory}). Thus $ R_R $ has finite hollow dimension. The module $M$, being a finitely generated $R$-module, is the epimorphic image of $R^n$ for some positive integer~$ n $. By~\cite[Proposition 2.42]{Facchini:ModuleTheory}, $ M $ has finite hollow dimension since the right $ R $-module $ R^n $ has finite hollow dimension.
\end{proof}

Since semiperfect rings are semilocal, this corollary
generalizes the well-known theorem on the existence of the decomposition as a direct sum of a projective submodule and a stable submodule for finitely generated modules over semiperfect rings \cite[Theorem 1.4]{Warfield:SerialRingsAndFinitelyPresentedModules}; see also \cite[Theorem 3.15]{Facchini:ModuleTheory} and \cite{Martsinkovsky:1TorsionOfFiniteModulesOverSemiperfectRings}.

\begin{remark}
Over a semiperfect ring $ R $, the Auslander-Bridger transpose, seen in the next section, induces a one-to-one correspondence between the isomorphism classes of finitely
presented stable right  and  left $R$-modules  by~\cite[Theorem 2.4]{Warfield:SerialRingsAndFinitelyPresentedModules}. Over any ring $R$, using again the Auslander-Bridger transpose, Facchini and Girardi obtain a one-to-one correspondence between the isomorphism classes of \emph{Auslander-Bridger right and left $ R $-modules} (which are finitely presented stable modules) introduced  in~\cite{Facchini-Girardi:AuslanderBridgerModules}; see also the monograph~\cite[Chapter 6]{Facchini:SemilocalCategoriesAndModulesWithSemilocalEndomorphism}. Denote by $\mathcal{P}$ the class consisting of projective modules that are finite direct sums of hollow projective modules (which are finitely generated by~\cite[Lemma 2.1]{Facchini-Girardi:AuslanderBridgerModules}). An \emph{Auslander–Bridger module} is defined as a stable module $M$ with a presentation $Q \rightarrow P \rightarrow M \rightarrow 0$, where $Q$ and $P$ are in $\mathcal{P}$.

In \cite{Facchini-Girardi:AuslanderBridgerModules}, it was shown that, for any ring $ R $, if a module $M$ is the epimorhic image of a module $P$ in $\mathcal{P}$, then $M=P^{\prime} \oplus N$, where $N$ is a stable submodule and $P^{\prime}$ is in $\mathcal{P}$; moreover, in such a decomposition, both of the submodules $P^{\prime}$ and $N$ are unique up to isomorphism~\cite[Proposition 3.5]{Facchini-Girardi:AuslanderBridgerModules}. If, in addition, there is a presentation $Q \rightarrow P \rightarrow M \rightarrow 0$ with $Q$ and $P$ in $\mathcal{P}$, then the submodule $N$ is an Auslander-Bridger module \cite[Corollary 3.8]{Facchini-Girardi:AuslanderBridgerModules}. The class of $R$-modules that have such presentations coincides with the class of all finitely presented $ R $-modules if and only if the class consisting of modules that are epimorphic images of modules in $\mathcal{P}$ coincides with the class of all finitely generated $ R $-modules if and only if the ring $ R $ is semiperfect~\cite[Lemma 6.7]{Facchini:SemilocalCategoriesAndModulesWithSemilocalEndomorphism}

Note that the existence of the decomposition of the module $M$ as a direct sum of a projective submodule and a stable submodule in the latter statements also follows from Theorem~\ref{Cor:Decomposition-hold-for-finite-udim-hdim}.
Indeed, $M$ has finite hollow dimension by~\cite[Proposition 2.42]{Facchini:ModuleTheory} since it is an epimorphic image of a module in $\mathcal{P}$, and modules in $\mathcal{P}$ have finite hollow dimension.
\end{remark}

\section{The Auslander-Bridger Transpose and Decompositions of Modules}\label{Sec:Decompositions_over_Semihereditary_Rings}

The Auslander-Bridger transpose functor $ \Transpose $ plays an important role in the representation theory of Artin algebras; see \cite[Section IV.1]{Auslander-Reiten-Smalo:RepresentationTheoryOfArtinAlgebras} and~\cite{Auslander-Bridger:StableModuleTheory}. It can be defined over any ring $R$; for details, see the monograph~\cite[Section 6.1, pp. 195--199]{Facchini:SemilocalCategoriesAndModulesWithSemilocalEndomorphism}. The Auslander-Bridger transpose is a duality   $\Transpose :  \underline{\bmod }$-$R \rightarrow R$-$\underline{\bmod }$ of the stable category $\underline{\bmod }$-$R$ of finitely presented right $R$-modules into the stable category $R$-$ \underline{\bmod } $ of finitely presented left $R$-modules. Here the stable category $ \underline{\bmod}$-$R$ is the factor category of the full subcategory $\bmod$-$R$ of the category $\operatorname{Mod}\text{-}R$ of all right $ R $-modules whose objects are all finitely presented right $R$-modules modulo the ideal of all morphisms that factor through a projective module, and similarly for $R$-$ \underline{\bmod} $.  For the stable category of modules, see~\cite[Section 4.11, p. 142]{Facchini:SemilocalCategoriesAndModulesWithSemilocalEndomorphism}. Similarly, one finds a functor $\operatorname{Tr}: R$-$\underline{\bmod} \rightarrow \underline{\bmod}$-$R$, and these two functors are {quasi-inverses of each other.} The properties shown in~\cite[Section 6.1, pp. 195--199]{Facchini:SemilocalCategoriesAndModulesWithSemilocalEndomorphism} and some results from~\cite[\S 5]{Sklyarenko:RelativeHomologicalAlgebra} are summarized below.

Let $M$ be a \emph{finitely presented} right $R$-module. Consider a \emph{projective presentation} of $ M $, that is, an exact sequence
\[
\gamma: \qquad \xymatrix@1{P_1 \ar[r]^f & P_0 \ar[r]^g & M \ar[r] & 0}
\]
where $P_0$ and $P_1$ are \emph{finitely generated projective} modules. Apply the functor
\[
(-)^*=\Hom_{R}(-, R): \text{Mod-} R\rightarrow R\text{-Mod} 
\]
to this presentation $ \gamma $:
\[
\xymatrix@1{ 0 \ar[r] & M^*=\Hom_R(M,R) \ar[r]^{g^*} &  P_0^*=\Hom_R(P_0,R) \ar[r]^{f^*}  & P_1^*= \Hom_R(P_1,R) }.
\]

Complete the right side of this sequence of \emph{left} $R$-modules by the module
\[
\Transpose_\gamma(M)=\Coker(f^*)=P_1^*/\Image(f^*)
\]
 to obtain the exact sequence
\begin{equation}\label{presentation of TrM}
\gamma^*: \qquad
\xymatrix@1{ P_0^* \ar[r]^{f^*}  &  P_1^* \ar[r]^-\sigma  & \Transpose_\gamma(M)  \ar[r] & 0,}
\end{equation}
where $\sigma$ is the canonical epimorphism. Since the modules $P_{0}^{*}$ and $P_{1}^{*}$ are finitely generated projective left $R$-modules, the exact sequence~(\ref{presentation of TrM}) is a projective presentation for the finitely presented \emph{left} $R$-module $\Transpose_\gamma(M)$, {called} the \emph{Auslander-Bridger transpose} of the finitely presented right $R$-module $M$ \emph{with respect to the projective presentation}~$\gamma$. If $\delta$ is another projective presentation of the finitely presented right $R$-module $M$, then $\Transpose_{\gamma}(M)$ and $\Transpose_{\delta}(M)$ are projectively equivalent, that is, \[\Transpose_{\gamma}(M)\oplus P\cong \Transpose_{\delta}(M)\oplus Q\] for some (finitely generated) projective modules $P$ and $Q$ as it follows from~\cite[\S 6.1, Proposition 6.1]{Facchini:SemilocalCategoriesAndModulesWithSemilocalEndomorphism}. Therefore, an Auslander-Bridger transpose of the finitely presented $ R $-module $ M $ is unique up to projective equivalence. So, we will use the notation $\Transpose(M)$  for the transpose. Moreover $\Transpose_{\gamma^{*}}(\Transpose_{\gamma}(M))\cong M$. If we drop the subscript for the presentations $\gamma^{*}$ and $\gamma$ in $\Transpose_{\gamma^{*}}(\Transpose_{\gamma}(M))$, then we can only say that $\Transpose(\Transpose(M))$ is projectively equivalent to $M$. Note that $\Transpose_{\gamma^*}(\Transpose_\gamma(M)) = \Coker(f^{**})$ is defined by the exact sequence:
 \[
\gamma^{**}: \qquad
\xymatrix@1{ P_1^{**} \ar[r]^{f^{**}}  &  P_0^{**} \ar[r]^-{\sigma'}  & \Transpose_{\gamma^*}(\Transpose_\gamma(M))  \ar[r] & 0,}
\]
where $\sigma'$ is the canonical epimorphism. On the other hand, applying the functor $(-)^*$ to the exact sequence (\ref{presentation of TrM}), we obtain the following exact sequence:
\[
\xymatrix@1{
0 \ar[r] &  (\Transpose_\gamma(M))^*  \ar[r]^<<<<<{{\sigma}^*}
& P_1^{**} \ar[r]^{f^{**}}  & P_0^{**}
}
.\]
Since we have natural isomorphisms $P\cong P^{**}$  for every finitely generated projective $R$-module~$P$, we obtain  $(\Transpose_\gamma(M))^* \cong \Image(\sigma^*)=\Ker(f^{**}) \cong\Ker(f)$. This proves:

\begin{proposition}
\label{proposition:when-is-pd(M)-at-most-1}
\cite[Lemma 6.1-(2)]{Hugel-Bazzoni:TTFTriplesInFunctorCategories} For a finitely presented $R$-module $M$, $\pd(M) \leq 1$ if and only if there exists a   presentation
\[
\gamma: \qquad \xymatrix@1{P_1 \ar[r]^f & P_0 \ar[r]^g & M \ar[r] & 0}
\]
of $M$ such that $(\Transpose_\gamma(M))^*=0$.
\end{proposition}

The properties of the Auslander-Bridger transpose that we shall use from \cite[\S 5]{Sklyarenko:RelativeHomologicalAlgebra} {are} summarized in the below theorem.

\begin{theorem}\label{theorem:Auslander-Brdiger-transpose-elementary-properties}
  \cite[Proposition 5.1, Remarks 5.1 and 5.2]{Sklyarenko:RelativeHomologicalAlgebra}
    Let $M$ be a \emph{finitely presented}  $R$-module and let $\gamma $ be a presentation of $M$:
    \[
\gamma: \qquad \xymatrix@1{P_1 \ar[r]^f & P_0 \ar[r]^g & M \ar[r] & 0}
\]
\begin{enumerate}\parantezicindeenumi
  \item For every $R$-module $N$, there is a monomorphism $\mu_N : \Ext^1_R(M,N)\rightarrow N \otimes_R \Transpose_\gamma(M)$ and for every left $R$-module $N$, there is an epimorphism
  $\varepsilon_N: \Hom_R(\Transpose_\gamma(M),N) \rightarrow \Tor_1^R(M,N)$.
  Both are natural in $N$.
  \item If  $\pd(M)\leq 1$, then the map $f:P_1\rightarrow P_0$ in the above presentation $\gamma$ can be taken to be a monomorphism and in this case the monomorphism  {$\mu_N$ and the epimorphism $\varepsilon_N$} become isomorphisms. Moreover by taking $N=R$, we obtain
      \[
      \Transpose_\gamma(M) \cong  \Ext^1_R(M,R)
      \qquad \text{ and } \qquad
      (\Transpose_\gamma(M))^*=\Hom_R(\Transpose_\gamma(M),R) = 0
      \]
      for the presentation $\gamma$ of $M$ where the map $f:P_1\rightarrow P_0$ is a monomorphism.
  \item If $\pd(M)\leq 1$, then $\Transpose(M)$ is projectively equivalent to $\Ext^1_R(M,R)$.
  \item If $M^*=\Hom_R(M,R)=0$, then
    {$ f^*: P_0^* \to P_1^* $ is a monomorphism in the presentation $ \gamma^* $ of $ \Transpose_\gamma(M) $ in (\ref{presentation of TrM}), $\pd(\Transpose_\gamma(M))\leq 1$ and
      \[
      M\cong \Transpose_{\gamma^*}(\Transpose_\gamma(M)) \cong  \Ext^1_R(\Transpose_\gamma(M),R).
      \]}
\item If $M$ is \emph{not} projective, then  $\Transpose_\gamma(M)\neq 0$.

\end{enumerate}
\end{theorem}

Let $R$ and $S$ be rings, and $N$ be an $S$-$R$-bimodule. For each $ s\in S $, let ${ }_s f$ be the right $R$-module endomorphism $N \rightarrow N$ with ${ }_s f(x)=s x$, for $x \in N$. Moreover, let ${ }_S \mathrm{Mod}_R^R$ denote the category with objects being all $S$-$R$-bimodules, while morphisms being right $R$-module homomorphisms. The symbol ${Ab}$, as always, denotes the category of Abelian groups.

\begin{lemma}\label{Lemma:ModuleStructureOnExtHomeTensoretc}
\begin{enumerate}\parantezicindeenumi
 \item Let $F:{ }_S \mathrm{Mod}_R^R \rightarrow A b$ be a covariant additive functor. Then, for any $N \in O b  { }_S \mathrm{Mod}_R^R$, $F(N)$ can be equipped with the structure of a left $S$-module by defining:
\[
s g=F\left({ }_s f\right)(g),
\]
for any $s \in S$ and $g \in F(N)$.

\item Let $F$ and $F^{\prime}$ be covariant additive functors ${ }_S \mathrm{Mod}_R^R \rightarrow A b$, and $\mu: F \rightarrow F^{\prime}$ be a natural transformation. Then, for every $N \in O b{ }_S \mathrm{Mod}_R^R$,  $\mu_N$ is a left $S$-module homomorphism.
\end{enumerate}
\end{lemma}
\begin{proof}
 The claim (i) can be easily verified. The claim (ii) immediately follows from the commutativity of the diagram
\[
\xymatrix{
F(N) \ar[r]^{\mu_N} \ar[d]_{F\left({ }_s f\right)} & F'(N)
\ar[d]^{ F'\left({ }_s f\right) } \\
F(N) \ar[r]_{\mu_N} &  F'(N)
}
\]
for every $ s\in S $.
\end{proof}

One can formulate the ``right $R$-'' version of this lemma. 

For an $R$-module $ M_R $, applying Lemma \ref{Lemma:ModuleStructureOnExtHomeTensoretc}-(i) to the functor
 $ F= \Ext^1_R(M,-):{ }_S \mathrm{Mod}_R^R \rightarrow A b $, we obtain a left $ S $-module structure in the abelian group $ \Ext^1_R(M,N) $ for every $S$-$R$-bimodule $ N $ and this is indeed the well known $ S $-module structure for it; see \cite[Theorem V.2.1 and \S V.3, p.144, Eqn. (3.4)]{Maclane:Homology}. Similarly, we obtain that
the structures of left/right modules provided by Lemma \ref{Lemma:ModuleStructureOnExtHomeTensoretc} and its ``right $R$-'' version coincides with the well-known structure on $\operatorname{Hom}_R(M, N)$, $ \Ext^1_R(M,N) $, $ \operatorname{Tor}_1^R(M, N)$ and
 $ M \otimes_R N$ for suitably choosen (left/right/bi)modules $ M $ and~$ N $ (see \cite[\S V.3]{Maclane:Homology} and \cite[\S 2.6, the paragraph after Definition 2.6.4]{Weibel:HomologicalAlgebra}).

\begin{proposition}
All group isomorphisms mentioned in
 Theorem \ref{theorem:Auslander-Brdiger-transpose-elementary-properties}-(ii) and (iv) are left or right module isomorphisms.
\end{proposition}
\begin{proof}
 For a finitely presented  $R$-module $ M_R $, applying Lemma \ref{Lemma:ModuleStructureOnExtHomeTensoretc} to the functors
  \[ F= \Ext^1_R(M,-):{ }_S \mathrm{Mod}_R^R \rightarrow A b
  \qquad \text{ and } \qquad
    F'= -\otimes_R \Tr_\gamma(M):{ }_S \mathrm{Mod}_R^R \rightarrow Ab, \]
and the natural transformation $ \mu:F\to F'$  given in Theorem \ref{theorem:Auslander-Brdiger-transpose-elementary-properties}-(i), we obtain that for every $ S $-$R $-bimodule $ { }_S N_R$, the monomorphism $\mu_N : \Ext^1_R(M,N)\rightarrow N \otimes_R \Transpose_\gamma(M)$ is a left $ S $-module homomorphism. Hence if we take $ N=  { }_R R_R$, we obtain a left $ R $-module homomorphism $\mu_R : \Ext^1_R(M,R)\rightarrow R \otimes_R \Transpose_\gamma(M)$. We also have the natural isomorphism $ R \otimes_R \Transpose_\gamma(M) \cong   \Transpose_\gamma(M) $ of left $ R $-modules \cite[\S V.3, Eqn. (3.9)]{Maclane:Homology}. This gives us the isomorphism $  \Transpose_\gamma(M) \cong  \Ext^1_R(M,R) $ of left $ R $-modules in Theorem~\ref{theorem:Auslander-Brdiger-transpose-elementary-properties}-(ii).

 Similarly, one shows that the epimorphism $\varepsilon_N : \Hom_R(\Transpose_\gamma(M),N) \rightarrow \Tor_1^R(M,N)$ is a right $ S $-module homomorphism for an $ R $-$S $-bimodule $ { }_R N_S$. Hence if we take $ N=  { }_R R_R$, we obtain a right $ R $-module homomorphism $\varepsilon_R : (\Transpose_\gamma(M))^*= \Hom_R(\Transpose_\gamma(M),R) \rightarrow \Tor_1^R(M,R)=0$.
\end{proof}

If $P$ is a nonzero projective $R$-module, then $P^*\neq 0$ {by the Dual Basis Lemma \cite[Theorem 5.4.2]{Kasch:ModulesAndRings}}. This gives us the following proposition, which
we shall frequently use:

\begin{proposition}\label{Proposition:Zero-dual-implies-stable} \cite[Lemma 2.6, (1)$\implies$(3)]{Martsinkovsky:TheStableCategoryOfALeftHereditaryRing}
  If $M^*=0$ for a right $R$-module $M$, then $M$ is stable.
\end{proposition}
\begin{proof}
  If $M=P\oplus N$ for submodules $P$ and $N$ of $M$ where $P$ is a projective module, then $0=M^*\cong P^* \oplus N^*$ gives $P^*=0$ which implies $P=0$ by the above observation.
\end{proof}

Theorem~\ref{theorem:Auslander-Brdiger-transpose-elementary-properties}-(ii) and Proposition \ref{Proposition:Zero-dual-implies-stable} immediately imply

\begin{proposition}
\label{prop:Ext-M-R-is-stable-if-pdM-leq-1}
  If  $M$ is a \emph{finitely presented} (right) $R$-module with $\pd(M)\leq 1$, then the \emph{left} $R$-module $\Ext^1_R(M, R)$ is finitely presented and stable.
\end{proposition}

Since the projective dimension of $\Transpose_{\gamma}(M)$ does not depend on a presentation $\gamma$, we will use the notation
$\pd(\Transpose(M))$ for it. Similarly, we use the symbol  $  \Ext^1_R(\Transpose(M), R)$.

\begin{lemma}\label{theorem 3.6}
Let $M$ be a finitely presented $R$-module.
\begin{enumerate}\parantezicindeenumi
  \item If $M^*=0$, then $\pd(\Transpose(M))\leq 1$.

  \item If $\pd(\Transpose(M))\leq 1$, then we have:
  \begin{enumerate}
    \item $M^*$ is a projective and finitely generated \emph{left} $R$-module.
    \item $M^{**}$ is a projective and finitely generated (right) $R$-module.
    \item {The dual of the $R$-module $ \Ext^1_R(\Transpose(M), R)$ is zero, and hence it is stable. }

      \item $M\cong M^{**} \oplus \Ext^1_R(\Transpose(M),R)$ gives a decomposition of the (right) $R$-module $M$ as a direct sum of a projective $R$-module and a stable $R$-module.
  \end{enumerate}

  \item If $\pd(\Transpose(M))\leq 1$ \emph{and} $M$ is stable, then $M^*=0$.
\end{enumerate}
\end{lemma}
\begin{proof}
 The claim (i) immediately follows from Theorem~\ref{theorem:Auslander-Brdiger-transpose-elementary-properties}-(iv). See the proof of \cite[Lemma 6.1-(1)]{Hugel-Bazzoni:TTFTriplesInFunctorCategories}.
 For (ii) suppose
   $\Transpose(M)$ has projective dimension at most $1$.
  Let $\gamma $ be a presentation of $M$:
    \[
\gamma: \qquad \xymatrix@1{P_1 \ar[r]^f & P_0 \ar[r]^g & M \ar[r] & 0}.
\]
 Then we have the following exact sequence
  \[
  \xymatrix@1{
0 \ar[r] & M^* \ar[r]^{g^*} &   P_0^* \ar[r]^{f^*}  &  P_1^* \ar[r]^-\sigma  & \Transpose_\gamma(M)  \ar[r] & 0
}
  \]
and so $M^*$ is projective. Indeed $M^* \cong \Image(g^*) = \Ker (f^*)$ is a direct summand of $P_0^*$ (since $\Image(f^*)$ is projective as $\pd(\Transpose(M))\leq 1$). Since $P_0^*$ is finitely generated, so is its direct summand $ \Image(g^*) \cong M^* $. By \cite[Proposition 5]{Masek:GorensteinDimensionAndTorsionOfModulesOverCommutativeNoetherianRings}, we have the following exact sequence of right $ R $-modules:
\[
\xymatrix@1{ 0 \ar[r]  & \Ext^1_R(\Transpose(M),R) \ar[r] & M  \ar[r]^<<<<<{\sigma_M} &  M^{**} \ar[r]  & \Ext^2_R(\Transpose(M),R)   \ar[r] & 0},
\]
where $\sigma_M:M \rightarrow M^{**}$ is the natural map into the double dual defined by
$ \sigma_M (m)(f)=f(m) $ for all $ m\in M $ and $  f: M\rightarrow R $ in $  M^* $. The last term $ \Ext^2_R(\Transpose(M),R)  $ is zero since   $\pd(\Transpose(M))\leq 1$. Since $M^*$ is finitely generated and projective, $M^{**}$ is also projective and so the exact sequence
    \[
\xymatrix@1{ 0 \ar[r]  & \Ext^1_R(\Transpose(M),R) \ar[r] & M  \ar[r]^<<<<<{\sigma_M} &  M^{**} \ar[r]  &  0},
    \]
splits. This implies that $M\cong M^{**} \oplus \Ext^1_R(\Transpose(M),R)$. Let $N=\Transpose(M)$. By the left version of Proposition~\ref{prop:Ext-M-R-is-stable-if-pdM-leq-1}, for the finitely presented left $R$-module $N$ that satisfies $\pd(N)\leq 1$, we have that $\Ext^1_R(N, R)$ is a finitely presented (right) $R$-module, and it is stable since its dual is zero. For (iii), if we further assume that $M$ is stable, then the projective direct summand $M^{**}$ must be zero. For the projective finitely generated left $R$-module $M^*$, this gives $M^* \cong M^{***}=0^*=0$. Hence $M^* = 0$.
\end{proof}

Lemma~\ref{theorem 3.6}-(i,iii) and Proposition~\ref{Proposition:Zero-dual-implies-stable} immediately imply
\begin{corollary}
For a finitely presented $R$-module $M$,
{$M^*=0$ if and only if $M$ is stable and $\pd(\Transpose(M))\leq 1$.}
\end{corollary}

\begin{remark}
The ``only if'' part of the statement ``$\pd(\Transpose(M))\leq 1$  if and only if $M^*=0$ for a finitely presented $R$-module $M$'' in  \cite[Lemma 6.1-(1)]{Hugel-Bazzoni:TTFTriplesInFunctorCategories} is not correct.
Clearly, it is false if $ M $ is a nonzero finitely generated projective module  or, more generally, $M=P \oplus L$, where $P$ is a nonzero finitely generated projective module, while $L$ is a finitely presented module with $L^*=0$. In this case, $ \Transpose(M) $ is projectively equivalent to $ \Transpose(L) $ and since $ L^*=0 $, we have $ \pd(\Transpose(M)) =   \pd(\Transpose(L)) \leq 1 $ by Theorem~\ref{theorem:Auslander-Brdiger-transpose-elementary-properties}-(iv), but $M^* \cong P^* \neq 0$.
\end{remark}

If $R$ is a left semihereditary ring, then every finitely presented left $R$-module has projective dimension $\leq 1$. Hence, for every finitely presented (right) $R$-module $M$, the finitely presented left $R$-module $\Transpose(M)$ satisfies $\pd(\Transpose(M))\leq 1$. Lemma~\ref{theorem 3.6} implies

\begin{theorem}\label{Ex:left semihereditary right module}
If $R$ is a \emph{left semihereditary} ring and $M$ is a \emph{finitely presented (right)} $R$-module, then $M=P\oplus N$ for some \emph{projective} submodule $P$ of $M$ and \emph{stable} submodule $N$ of~$M$ with
$
P\cong M^{**} $ and $ N\cong \Ext^1_R(\Transpose(M),R)
$.
\end{theorem}

\section{Examples where the Decomposition Fails}\label{Sec:Examples_where_the_decomposition_fails}

In this section, we give examples of modules that have no decomposition as a direct sum of a projective submodule and a stable submodule.

By the result in~\cite{Zangurashvili:AStructureTheoremForLeftModulesOverLeftHereditaryLeftPerfectRightCoherentRings} mentioned in the introduction, over a right hereditary ring,  every (right) $ R $-module can be decomposed as a direct sum of a projective submodule and a stable submodule if and only if the ring $ R $ is right perfect and left coherent. To construct {examples of modules} for which the decomposition fails, we shall give below another proof of the `only if' part of that result using the relationship between torsionless modules and projective modules, and the result from Chase~\cite[Theorem 3.3]{Chase:DirectProductsOfModules} characterizing the rings over which every direct product of projective modules is projective (or, {equivalently, every direct product of copies of the ring, viewed as a right module, is projective}) as the right perfect and left coherent rings.

Recall that an $R$-module $M$ is said to be \emph{torsionless} if $M$ can be embedded as an $R$-submodule into a direct product $\prod_{i\in I} R_R$ for some index set $I$. By~\cite[Remark 4.65(a)]{Lam:LecturesOnModulesAndRings}, an $R$-module $M$ is torsionless if and only if for every $m \neq0$ in $M$, there exist a homomorphism $f\in M^{\ast}=\Hom_{R}(M,R)$ such that $f(m)\neq 0$. Thus, an $R$-module $M$ is torsionless if and only if the natural map $ {\sigma_M} : M  \rightarrow M^{\ast\ast} $, defined {by $ {\sigma_M}(m)(f)=f(m) $, for all $ m\in M $ and $ f\in M^* $}, is injective.

Every submodule of a free $R$-module is clearly torsionless. So every projective module is torsionless since it is a direct summand of a free module. The converse holds, that is, all torsionless (right) $R$-modules are projective, only if $R$ is a right perfect and left coherent ring; this follows from the above mentioned characterization~\cite[Theorem 3.3]{Chase:DirectProductsOfModules} of Chase. 

As shown in~\cite[Lemma 2.6]{Martsinkovsky:TheStableCategoryOfALeftHereditaryRing}, if $R$ is a right hereditary ring, then for a right $R$-module $M$, $M^*=0$ if and only if $M$ is stable. The same equivalence also holds for finitely generated modules over a right semihereditary ring.

\begin{theorem}
\label{theorem:hereditary-not-right-perfect-or-left-coherent}
 \parantezicindeenumi
If a ring $R$ is not right perfect or left coherent, then there exists a torsionless $R$-module that is not projective. A torsionless $ R $-module $ M $ that is not projective does not  have a decomposition into the direct sum of a projective submodule and a stable submodule {in either of the following cases}:
\begin{enumerate}
 \item $R$ is a right hereditary ring;  
  \item $R$ is a right semihereditary ring and $ M $ is finitely generated.
\end{enumerate}

\end{theorem}
\begin{proof} 
\parantezicindeenumi
\begin{enumerate}
\item Suppose for the contrary that a torsionless but not projective $R$-module $M$ has a decomposition $M=P\oplus N$ with a projective module $P$ and a stable module $N$. Since $M$ is not projective, $N\neq 0$. Since $M$ is torsionless, its nonzero submodule $N$ is also torsionless, and hence $N^\ast \neq 0$ by~\cite[Remark 4.65(a)]{Lam:LecturesOnModulesAndRings}. Then $N$ is not a stable module by~\cite[Lemma 2.6]{Martsinkovsky:TheStableCategoryOfALeftHereditaryRing}, contradicting the assumption.
{\item {The proof in (i) extends to the semihereditary case} since  $ N \cong M/P$ is  finitely generated whenever $ M $ is finitely generated.}
\end{enumerate}
\end{proof}

{Below we give examples of right hereditary but not right perfect rings mentioned in Theorem~\ref{theorem:hereditary-not-right-perfect-or-left-coherent}.}

\begin{example}{(i) The ring $\Z$ of integers is a hereditary Noetherian commutative domain and by~\cite[Theorem 23.24]{Lam:AFirstCourseInNoncommutativeRingsSecondEdition}, it is not a perfect ring. The $\Z$-module $ M= \prod\limits_{i=1}^{\infty}\Z $ is obviously torsionless. But, as is well-known, it is not projective (see, e.g., \cite[Example 2.8]{Lam:LecturesOnModulesAndRings}).
\\
(ii) A right hereditary right perfect ring is also left hereditary~\cite[Corollary 2]{Small:SemihereditaryRings}. Hence, a right hereditary ring that is not left hereditary is not right perfect. For example, the triangular ring
$R=\left[
\begin{array}{cc}
Z & \Q \\
0 & \Q \\
\end{array}
\right]$
is right hereditary but not left hereditary \cite[Small's Example 2.33]{Lam:LecturesOnModulesAndRings}.}
\end{example} 

The modules for which the decomposition into projective and stable submodules fails, in the  above examples, are not finitely generated. We now construct finitely generated modules for which the decomposition fails.

A ring $R$ is called \emph{right Baer} (resp., \emph{left Baer}) if every right (resp., left) annihilator of every subset of $R$ is of the form $eR$  (resp., $Re$) for some idempotent $e$ in~$R$. Similarly, $R$ is called \emph{right Rickart} (resp., \emph{left Rickart}) if the right (resp., left) annihilator of every element of $R$ is of the form  $eR$  (resp., $Re$) for some idempotent $e$ in~$R$ (see ~\cite[Section 7D]{Lam:LecturesOnModulesAndRings}).

Clearly, a right Baer (resp., left Baer) ring is always a right Rickart (resp., left Rickart) ring. A ring $R$ is right Baer if and only if it is left Baer~\cite[Proposition 7.46]{Lam:LecturesOnModulesAndRings}. Furthermore, a ring $R$ is right Rickart if and only if every principal right ideal in $R$ is projective~\cite[Proposition 7.48]{Lam:LecturesOnModulesAndRings}. Therefore, the right semihereditary (resp., left semihereditary) rings are right Rickart (resp., left Rickart).

\begin{theorem} \label{Ex:finitelygenerated torsionless} If $R$ is a right semihereditary ring that is not a right Baer ring, then there exists a cyclic torsionless $R$-module $M$ which cannot be decomposed as $M=P\oplus N$ for some submodules $P$ and $N$ of $M$ such that $P$ is projective and $N$ is stable.
\end{theorem}
\begin{proof}
Let $R$ be a right semihereditary ring that is not a right Baer ring. Then there exists a subset $S$ of $R$ such that its right annihilator \[I=\rann_{R}(S)=\{r \in R \mid sr=0\text{ for all } s \in S\},\] is not a direct summand of $R_R$. Consider the element $ a = (s)_{s \in S} \in \prod_{s \in S} R_R$. The cyclic $R$-module $a R \cong R/I $ is torsionless, but it is not projective since $ I $ is not a direct summand of $R_R$. By Theorem~\ref{theorem:hereditary-not-right-perfect-or-left-coherent}-(ii), $ aR $ cannot be decomposed as $aR = P \oplus N$ with  projective $P$ and  stable $N$.
\end{proof}

Recall that a ring $R$ is called a \emph{von Neumann regular ring} if for every element $a\in R$, there exists $x \in R $ such that $axa=a$.
\begin{example}~\cite[Chase's Example 2.34]{Lam:LecturesOnModulesAndRings} Let $ S $ be a von Neumann regular ring with an ideal $I$ such that $I$ is not a direct summand of $S_S$ as a right $ S $-submodule. {For instance, any commutative nonsemisimple von Neumann regular ring S has such an ideal (for example, an infinite product of fields is such a ring).} Let  $R=S/I$, and view $R$ as an $R\text{-}S$-bimodule. Consider the triangular matrix ring $T=\left[
\begin{array}{ll}
  R & R \\
  0 & S \\
\end{array}\right]$. Then $T$ is left semihereditary but not right semihereditary~{\cite[Example 2.34]{Lam:LecturesOnModulesAndRings}}, and it is not right Rickart. Indeed, in the proof that $T$ is not right semihereditary in \cite[Example 2.34]{Lam:LecturesOnModulesAndRings}, it has been shown that there exists a principal right ideal of $ T $ that is not a projective $ T $-module. By~\cite[Proposition 7.48]{Lam:LecturesOnModulesAndRings}, this is equivalent to $R$ being not right Rickart. Thus $T$ is a left semihereditary ring which is not right Rickart. Then $T$ is not right Baer and hence also $T$ is not left Baer by \cite[Proposition 7.46]{Lam:LecturesOnModulesAndRings}. {Therefore, the ring that is opposite to $T$ is right semihereditary but not right Baer.}
\end{example}

We shall now see another class of finitely generated {modules} for which the decomposition fails.
We shall construct cyclic modules over \emph{right semiartinian right V-rings that are not semisimple} for which the decomposition fails. A ring $R$ is called \emph{right semiartinian} (resp., \emph{left semiartinian}) if every nonzero right (resp., left) $R$-module has a simple submodule (equivalently, $ \Soc(M) $ is an essential submodule of $ M $ for every nonzero right (resp., left) $R$-module $ M $); it is called semiartinian if it is both left and right semiartinian. A ring $R$ is called a \emph{right V-ring} if every simple right $R$-module is injective. Right semiartinian right V-rings belong to a special class of von Neumann regular rings; see~\cite[Sections 6.1 and 16, Theorem 16.14]{Jain-et-al:CyclicModulesAndTheStructureOfRings} and~\cite{Baccella:SemiArtinianVRingsAndSemiArtinianVonNeumannRegularRings},  where they are called \emph{right SV-rings}. A well-known example of such rings is the following.

\begin{example}
Let $F$ be a field and $V_F$ be an infinite dimensional vector space over $F$. Set $T=\Endomorphism_F(V_F)$ and $S= \{ f\in T \mid \dim_F (\Image(f))<\infty \}  $ (which is an ideal of $ T $). Let $R$ be the subring of $T$ generated by $S$ and the scalar transformations $ d 1_{V} $ for all $d \in F$ which form a subring of $ T $ isomorphic to $ F $ and identified with $ F $; so we write $ R=S+F $. Clearly there is a group isomorphism $R / S \cong F$. Due to this isomorphism, $F$ can be turned into both left and right $R$-modules. These modules are simple. The ring $ R $ is von Neumann regular (in fact, \emph{unit-regular}, that is, for every element $a\in R$, there exists a unit $u \in R $ such that $aua=a$ \cite[Proofs of Examples 6.19 and 5.15]{Goodearl:VonNeumannRegularRings}). By~\cite[Example 5.14]{Cozzens-and-Faith:SimpleNoetherianRings}, $R$ is a right V-ring which is not a left V-ring.

See \cite[\S 10,11]{Lam:AFirstCourseInNoncommutativeRingsSecondEdition} for the definitions of the terms in what follows.
It is easily seen that $ V $ is a faithful simple left $ R $-module.
Since the subring  $ R $ of $ T $ contains~$ S $, it is a dense subring of~$ T $.
Then by \cite[Theorem 11.20]{Lam:AFirstCourseInNoncommutativeRingsSecondEdition},
  $ R $ is a left primitive ring  and $ \End_R(V)=F $. Since $ R $ is left primitive, it is also
  a prime (and so semiprime) ring by \cite[Proposition 11.6]{Lam:AFirstCourseInNoncommutativeRingsSecondEdition}, and so its left  and right socles  coincide by \cite[p. 175, last paragraph]{Lam:AFirstCourseInNoncommutativeRingsSecondEdition}: $ \Soc({}_R R) = \Soc(R_R) $.
  By \cite[Ex. 11.18]{Lam:ExercisesInClassicalRingTheorySecondEdition}
(or \cite[Theorem 2.1.25]{Rowen:RingTheoryStudentEdition}), $ \Soc({}_R R) $ consists of all finite rank linear operators in $ R $, that is, $\Soc({}_R R)=S $ (and so $ R $ is not semisimple).

 $ R $
 is
 right semiartinian. Indeed, if $ I $ is a proper right ideal of $ R $, then $ (I+S)/I \cong S / (S\cap I)   $.
  Therefore, if $ S $ is not a submodule of~$ I $, then $ (I+S)/I $ is a nonzero semisimple submodule of $ R/I $ since $ S=\Soc(R_R) $ is semisimple; if $ S\subseteq I \subsetneqq R $, then $ I=S $ since $ R/S $ is a simple $ R $-module and hence $ R/I=R/S $ is a simple $ R $-module. Since any nonzero right module contains a submodule that is isomorphic to $ R/I $ for some $ I $, $ R $ is a right semiartinian right V-ring that is not semisimple.
\end{example}

\begin{theorem} \label{Ex:Cyclicmodule} If $R$ is a right semiartinian right V-ring that is not a semisimple ring, then there exists a \emph{cyclic} (right) $R$-module $M$ that cannot be decomposed as $M=P\oplus N$ such that $P$ is projective and $N$ is stable.
\end{theorem}
\begin{proof}

Since $R$ is a right semiartinian ring that is not a semisimple ring, the nonzero right $R$-module $R/\Soc(R_R)$ must have a simple submodule $C/\Soc(R_R)$ where $\Soc(R_R)\subseteq C \subseteq R$. Let $c\in C\setminus \Soc(R_R)$. Then \[C/\Soc(R_R)=(c+\Soc(R_R))R=(cR+\Soc(R_R))/\Soc(R_R).\] Let $D=cR$. Note that $D$ is not semisimple; otherwise, $D=cR\subseteq \Soc(R_R)$, contradicting $c \notin \Soc(R_R)$.
Thus $\operatorname{Soc}(D) \neq D$. This implies that $\Soc (D)$ is not finitely generated. Indeed, if it were finitely generated, it would be a direct sum of finitely many simple right
$R$-modules, which are injective as $R$ is a right V-ring, and so $\Soc(D)$ would be injective, which would then be a direct summand of $D$. But this would lead to a contradiction since $\Soc (D)$ is an essential submodule of $D$ (as the ring $R$ is right semiartinian).

Consider the module $D/\Soc(D)$; it is simple since it is isomorphic to the simple module $ C/\Soc(R_R) $:
\[D/\Soc(D)=cR/\Soc(cR)=cR/cR\cap \Soc(R_R)\cong (cR+\Soc(R_R))/\Soc(R_R)= C/\Soc(R_R).\] {Therefore, $\Soc(D)$ is a maximal submodule of $D$. Since the semisimple module} $\Soc(D)$ is not finitely generated, there exists a decomposition $\Soc(D)=A\oplus B$ where both $A$  and $B$ are semisimple submodules of $D$ that are not finitely generated. Let $M=D/A$.

Suppose that $M=D/A=(P/A)\oplus (N/A)$ for some submodules $P$, $N$ of $D$ such that $A\subseteq P$, $N \subseteq D=cR\subseteq R$, where $P/A$ is projective and $N/A$ {is stable}.

Firstly, we must have $ P\neq D $; otherwise, $M=D/A=P/A$ would be projective, making $ A $  a direct summand of the cyclic module $P=D=cR$, and contradicting the fact that $ A $ is not finitely generated. Suppose that $P\supseteq \Soc(D)$. Since $\Soc(D)\subseteq P \subseteq D$ and $\Soc(D)$ is a maximal submodule of $D$, we must have either $P=\Soc(D)$ or $P=D$. As seen above $P\neq D$, and so we must have $P=\Soc(D)=A\oplus B$. In this sum, $B\cong  P/A$ is a cyclic $R$-module since it is a quotient of the cyclic $R$-module $D/A=cR/A$. But by our choice, $B$ is not finitely generated, leading to a contradiction. Thus $P\nsupseteq \Soc(D)=A\oplus B$. Given $A\subseteq P$ and $A\oplus B\nsubseteq P$, there exists a simple submodule $S\subseteq B$ such that $S \nsubseteq P$. Thus $S\cap P=0$ and it implies $\left((S\oplus A)/A\right)\cap \Soc(P/A)=0$. We have
 \[S\cong(S\oplus A)/ A \subseteq \Soc(M)=\Soc(P/A)\oplus \Soc(N/A).\]
 The {semisimple submodule  $\left((S\oplus A)/A\right) \oplus \Soc(P/A) $ of $ M $ must be a direct summand of $ \Soc(M) $. It follows that} $\Soc(M)= ((S\oplus A)/A) \oplus  \Soc(P/A) \oplus (U/A)$ for some submodule $U/A$ of $\Soc(M)$, where $A\subseteq U\subseteq D$. Then $(S\oplus A)/A$ is isomorphic to a direct summand of $\Soc(M)/\Soc(P/A)\cong \Soc(N/A)$. Therefore $N/A$ should have a simple submodule $T/A\cong (S\oplus A)/ A \cong S$. But $S$ is a simple $R$-submodule of $R_R$. Since $R$ is a right V-ring, $S$ is injective. Then it is a direct summand of $R_R$ and hence is projective. Thus $N/A$ has an injective and projective simple submodule $T/A \cong S$ which is a direct summand of $N/A$ (since $S$ is injective). This contradicts the assumption that $N/A$ {is stable}. Therefore, the cyclic $R$-module $M$ does not have a decomposition as a direct sum of a projective submodule and a stable submodule.
\end{proof}

Over a right semiartinian right V-ring $R$ that is not a semisimple ring, we cannot find a finitely presented $R$-module $M$ that does not have a decomposition as a direct sum of a projective submodule and a stable submodule. {This is} because such rings are von Neumann regular and every finitely presented module over a von Neumann regular ring $ R $ is projective. Indeed, by~\cite[Theorem 1.11]{Goodearl:VonNeumannRegularRings}, for each positive integer $ n$, each finitely generated submodule $K$ of the finitely generated free $R$-module $R^{n}$ is a direct summand of $R^{n}$, and so the finitely presented $R$-module $ R^{n}/K$ will be projective since it is isomorphic to a direct summand of the projective module~$ R^n $.

We have seen examples of finitely generated modules that have no decomposition as a direct sum of a projective submodule and a stable submodule. In the final example below, we obtain a \emph{finitely presented module} (indeed, a \emph{cyclically presented module}) \emph{over a commutative ring} that has no decomposition as a direct sum of a projective submodule and a stable submodule {(because it is \emph{not projective and has no nonzero stable submodule}). Moreover, {as the following lemma  shows}, it is \emph{not projectively equivalent to any stable module}.}
{
\begin{lemma}\label{lem:proj-equiv-stable}
If  a module $M$ is not projective and has no nonzero stable submodule, then $M$ is not projectively equivalent to any stable module.
\end{lemma}
\begin{proof} Suppose for the contrary that $M$ is projectively equivalent to a stable module $U$. Then there exist projective modules $P$ and $Q$ and an isomorphism $\psi : U \oplus P \longrightarrow M \oplus Q$.
Let $\pi_M : M \oplus Q \to M$ and $\pi_Q : M \oplus Q \to Q$ be the canonical projection maps, and let $i_U : U \to U \oplus P$ be the canonical inclusion map. Define
$f = \pi_M \circ \psi \circ i_U : U \longrightarrow M$.
Since $U$ is stable, the quotient $U / \Ker(f)$ is also stable. Moreover, the homomorphism $f$ induces an isomorphism
$U / \ker(f) \ \cong \ \operatorname{Im}(f) \ \subseteq \ M$.
By hypothesis, $M$ has no nonzero stable submodule, so $\operatorname{Im}(f) = 0$. Thus $f = 0$, and hence $\psi(U) \subseteq 0 \oplus Q$. Therefore
\[P \ \cong \ (U \oplus P)/U \ \cong \ (M \oplus Q) / \psi(U) \ \cong \ M \oplus \big(Q / \pi_Q(\psi(U))\big).\]
This shows that $M$ is isomorphic to a direct summand of the projective module $P$, and therefore $M$ is projective, contradicting our hypothesis. Hence $M$ cannot be projectively equivalent to a stable module.
\end{proof}}

\begin{example} \label{Ex:Cyclically_presented_module}
Let $\Z_{2}=\Z/2\Z=\{\overline{0},\overline{1}\}$. Consider the commutative ring
\[R=\prod \limits_{i=1}^{\infty}\Z_{2}\quad \mbox{and its ideal}\quad D= \bigoplus \limits_{i=1}^{\infty}\Z_{2}\subseteq R.\] Let $T$ be a maximal ideal of $R$ such that $D\subseteq T \subseteq R$. Then $S=R/T$ is a simple $R$-bimodule and $D$ annihilates $S$ from both sides, that is, $SD=DS=0$. Let $A$ be the following commutative matrix ring: 
\[A=\left[\begin{array}{lll}                                                R & &S \\                                            &\setminus &  \\
                      0 & & R\\
                    \end{array}\right]= \left\{\left[\begin{array}{ll}
                       r & s\\
                        0 & r \\
                       \end{array}\right]
               : r\in R, s\in S \right\}.\]
The ring $A$ is obviously isomorphic to the ring that is $R \times S$ as a group, where the multiplication is defined by
$(r,s)(r',s')=(rr', rs'+sr')$ for all $(r,s)$, $(r',s')\in R\times S$. (This ring construction is called idealization. See \cite[\S 1]{Nagata:LocalRings1962} and \cite{Anderson-and-Winders:IdealizationOfAModule} for the `principle of idealization' introduced by Nagata.) The maximal ideals of the ring $R\times S$ are of the form $B \times S$ where $B$ is a maximal ideal of $R$, and the Jacobson radical of $ R\times S  $ is $ \Jac(R)\times S $; see  \cite[Theorem 3.2-(1)]{Anderson-and-Winders:IdealizationOfAModule}. This implies that the Jacobson radical of the ring $A$ is
$
{J} =
\left[
\begin{array}{ll}
0 &S \\
0 & 0 \\
\end{array}
\right]
$
since the commutative ring $R=\prod \limits_{i=1}^{\infty}\Z_{2}$ has $\Jac(R)=0$ (indeed, for every $i\in \Z^{+}$, $\Z_{2}\times \Z_{2}\times \cdots\times Z_{2}\times 0 \times \Z_2 \times \cdots$ is a maximal ideal of $R$, where $0$ is in the $i$-th coordinate and all other coordinates are $\Z_2$). Clearly,
$  {J}=  \left[\begin{array}{ll}
0 & S\\
0 & 0 \\
\end{array}\right]
=
\left[\begin{array}{ll}
0 & s\\
0 & 0 \\
\end{array}\right] A $ for every nonzero element $ s   $ in the simple $R$-module $S$, and so the module $M=A/J$ is cyclically presented. Note that ${J}=\Rad(A_A)$ is a superfluous submodule of $A$ and so it cannot be a direct summand of $A$. Hence $M=A/{J}$ is not a projective $A$-module.

We will show that $ M $
has no nonzero stable submodule, and hence has no decomposition of the desired
kind. Assume $N=Y/{J}$ is a nonzero stable submodule of $ M $, where ${J}\subsetneqq Y \subseteq A$. Then $Y$ has an element $y=\left[
 \begin{array}{ll}
a & s \\
 0 & a \\
 \end{array}
\right]$ that is not in ${J}$, where $s\in S$ and $0\neq a =(a_i)_{i=1}^{\infty}\in R=\prod \limits_{i=1}^{\infty}\Z_{2}$. Since $0\neq a$, there exists $n\in \Z^{+}$ such that $a_n=\overline{1}$. {Let $z=(\overline{0},\cdots ,\overline{0},\overline{1},\overline{0},\cdots)\in D\subseteq R$ be the sequence whose $n$-th coordinate is $\overline{1}$ and all other coordinates~$\overline{0}$. Let
                      \[L=\left[
                        \begin{array}{lll}
                      0 \times \cdots\times 0 \times \Z_{2} \times 0\cdots  & & 0\\
                        &\setminus & \\
                    0& & 0 \times \cdots\times 0\times \Z_{2} \times 0 \cdots \\
                     \end{array}
\right]=\left\{ \left[
                                 \begin{array}{ll}
                                   0 & 0 \\
                                   0 & 0 \\
                                 \end{array}
                               \right],  \left[
                                 \begin{array}{ll}
                                   z & 0 \\
                                   0 & z \\
                                 \end{array}
                               \right]  \right\}  ,\] where the $n$-th coordinate in $ 0 \times \cdots\times 0 \times \Z_{2} \times 0\times \cdots$ is $\Z_2$ and all other coordinates are $0$.
Then
 $L$ is an ideal of $A$ since $DS=SD=0$ and $zr=rz=z $ or $0$  for every $ r\in R$.

Since $az=z$ and $sz\in SD=0$, we obtain
\[y\left[
                                 \begin{array}{ll}
                                   z & 0 \\
                                   0 & z \\
                                 \end{array}
                               \right]=\left[
                                 \begin{array}{ll}
                                   a & s \\
                                   0 & a \\
                                 \end{array}
                               \right]\left[
                                 \begin{array}{ll}
                                   z & 0 \\
                                   0 & z \\
                                 \end{array}
                               \right]=\left[
                                 \begin{array}{ll}
                                   az & sz \\
                                   0 & az \\
                                 \end{array}
                               \right]=\left[
                                 \begin{array}{ll}
                                   z & 0 \\
                                   0 & z \\
                                 \end{array}
                               \right].
                               \]
                               Thus
\[
\left[
                                 \begin{array}{ll}
                                   z & 0 \\
                                   0 & z \\
                                 \end{array}
                               \right]
                               = y \left[
                                 \begin{array}{ll}
                                   z & 0 \\
                                   0 & z \\
                                 \end{array}
                               \right] \in yA \subseteq Y
\]
since $ y\in Y $ and $ Y $ is a submodule of the right $ A $-module $ A_A $. Hence $ L $ is a submodule of $ Y $.}

Furthermore, $A_A=L\oplus C$ for \[C= \left[\begin{array}{lll}
 \Z_{2}\times\cdots \times \Z_{2}\times 0 \times \Z_{2} \times \cdots& & S \\
           & \setminus & \\
       0 & &\Z_{2} \times \cdots \times \Z_{2} \times 0 \times \Z_{2} \times \cdots \\       \end{array}\right],\] where in $ \Z_{2} \times \cdots \times \Z_{2} \times 0 \times \Z_{2} \times \cdots $ the $n$-th coordinate is $0$ and all other coordinates are $\Z_2$.
The module $L$ is projective since it is a direct summand of the right $A$-module $A_A$.

Since $L\oplus {J}\subseteq Y$, we obtain by the modular law that
\[
N=Y/{J}
=(Y/J) \cap  (A/J)
= (Y/J) \cap  [[(L\oplus {J})/{J} ]\oplus (C/{J})]
= \left[(L\oplus {J})/{J}\right]\oplus \left[(C/{J})\cap (Y/{J})\right]
.\]

Thus $ (L\oplus{J})/{J} $ is a direct summand of $N$ and $ (L\oplus{J})/{J}  \cong  L $ is a nonzero projective $A$-module. This contradicts with $N$ being a stable $A$-module. Therefore, $M=A/{J}$ has no decomposition as a direct sum of a projective submodule and a stable submodule. Moreover, it is {not projectively equivalent to any stable module} by Lemma \ref{lem:proj-equiv-stable}.
\end{example}

\section*{Acknowledgements}

We would like to thank Noyan Er for his valuable contributions and fruitful discussions for the problems considered in this article about finding examples of modules that are either finitely generated or finitely presented modules and that cannot be decomposed as a direct sum of a projective module and a stable module; the important examples in Theorem \ref{Ex:Cyclicmodule} and Example~\ref{Ex:Cyclically_presented_module}  have been found by him.

Most of the results are from the
master thesis \cite{GulizarGunay:OnTheAuslanderBridgerTranspose}
of the first author.

The second author would also like to thank Kulumani Rangaswamy, Murat \"{O}zayd{\i}n and Pere Ara for the fruitful discussion about decompositions of finitely presented modules over Leavitt path algebras in the workshop \cite{CIMPA2015}. Pere Ara has referred in his article~\cite{AraAndBrustenga:TheRegularAlgebraOfAQuiver} to a result in the book \cite{Luck:L2invariants} for some classes of semihereditary rings. He suggested that the decomposition of a finitely presented module as a direct sum of a projective submodule and a stable submodule should hold over semihereditary rings. The results given in \cite[Theorem~6.7]{Luck:L2invariants}  motivated the proof of the decomposition result for finitely presented (right) $R$-modules over a \emph{left} semihereditary ring given in Section~\ref{Sec:Decompositions_over_Semihereditary_Rings} (Theorem~\ref{Ex:left semihereditary right module}).

The authors would like to thank the referee for the valuable comments and suggestions, which have helped to improve the clarity and presentation of the paper.

\bibliographystyle{abbrvnat}
\bibliography{kaynaklar}

\end{document}